\documentclass[alphabetic]{amsart}

\usepackage{enumerate, longtable}
\usepackage{amsmath, amscd, amsfonts, amsthm, amssymb, latexsym, comment, stmaryrd, graphicx, dsfont, amsaddr, mathtools}
\usepackage{xcolor}
\usepackage[all]{xy}
\usepackage{fullpage}
\usepackage{tikz}
\usepackage[
        colorlinks,
        linkcolor=red,  citecolor=blue,
        backref
]{hyperref}

\usepackage{amstext} 
\usepackage{array}
\usepackage{cleveref}

\newtheorem{claim}{Claim}[section]

\newtheorem{conj}[claim]{Conjecture} 
\newtheorem{thm}{Theorem}
\newtheorem{prop}[claim]{Proposition}
\newtheorem{cor}[claim]{Corollary}
\newtheorem{lem}[claim]{Lemma}

\theoremstyle{definition}

\newtheorem{defi}[claim]{Definition} 
\newtheorem{rem}[claim]{Remark}
\newtheorem{ex}[claim]{Example}

\numberwithin{equation}{section}

\newcommand{\case}[1]{\vspace{0.3cm}\noindent\framebox{#1.}}

\newcommand{\F}{{\mathbb F}}
\newcommand{\Z}{{\mathbb Z}}
\newcommand{\Q}{{\mathbb Q}}
\newcommand{\Qb}{{\overline{\Q}}}

\newcommand{\C}{{\mathbb C}}

\newcommand{\N}{{\mathbb N}}

\newcommand{\Gal}{{\mathrm{Gal}}}
\newcommand{\disc}{{\mathrm{disc}}}

\renewcommand{\Cup}{\bigcup}

\newcommand{\isom}{{\xrightarrow{\sim}}}

\newcommand{\lcm}{{\mathrm{lcm}}}

\numberwithin{equation}{section}

\author{Alexei Entin}
\address{Raymond and Beverly Sackler School of Mathematical Sciences, Tel Aviv University, Tel Aviv 69978, Israel}
\email{aentin@tauex.tau.ac.il}
\title{Lower bounds on the least common multiple of a polynomial sequence and its radical}

\begin{document}
\maketitle
\begin{abstract}
    Cilleruelo conjectured that for an irreducible polynomial $f \in \mathbb{Z}[X]$ of degree $d \geq 2$, denoting $$L_f(N)=\mathrm{lcm}(f(1),f(2),\ldots f(N))$$ one has $$\log L_f(n)\sim(d-1)N\log N.$$ He proved it in the case $d=2$ but it remains open for every polynomial with $d>2$.
    While the tight upper bound $\log L_f(n)\lesssim (d-1)N\log N$ is known, the best known general lower bound due to Sah is $\log L_f(n)\gtrsim N\log N.$

    We give an improved lower bound for a special class of irreducible polynomials, which includes the decomposable irreducible polynomials $f=g\circ h,\,g,h\in\Z[x],\deg g,\deg h\ge 2$, for which we show $$\log L_f(n)\gtrsim \frac{d-1}{d-\deg g}N\log N.$$

    We also improve Sah's lower bound $\log\ell_f(N)\gtrsim \frac 2dN\log N$ for the radical $\ell_f(N)=\mathrm{rad}(L_f(N))$ for all irreducible $f$ with $d\ge 3$ and give a further improvement
    for polynomials $f$ with a small Galois group and satisfying an additional technical condition, as well as for decomposable polynomials. 
\end{abstract}

\section{Introduction}
The study of the least common multiple of an integer polynomial sequence goes back to Chebychev, who observed that  $\log\lcm\left(1,2, ,\ldots, N\right) \sim N$ is equivalent to the prime number theorem and gave lower and upper bounds for this quantity. This problem later inspired a more general problem of studying the least common multiple of polynomial sequences. For a linear polynomial $f(x) = qx + a,\ k,a\in\Z,\ q>0, q+a > 0$ it was observed by Bateman and Kalb \cite{BKS02}, that
$\textrm{log lcm}\left(f\left(1\right),f(2),\ldots, f\left(N\right)\right) \sim c_f N$ as $N\to\infty$, where $c_f = \frac{q}{\varphi(q)} \sum_{1 \leq m \leq q,(m,q) = 1} \frac{1}{m}$, which is a consequence of the prime number theorem for arithmetic progressions. 

Cilleruelo \cite{Cil11} studied the quantity
\begin{equation}\label{eq: cilleruelo}L_f(N)=\textrm{lcm}\left(f\left(1\right),f(2),..., f\left(N\right)\right)\end{equation} for an irreducible polynomial $f\in\Z[x]$ of degree $d\ge 2$. In the case $d=2$ he proved the asymptotic $\log L_f(N)\sim N\log N$ and made the following
\begin{conj}\label{conj: cilleruelo}Let $f\in\Z[x]$ be a fixed irreducible polynomial of degree $d\ge 2$. Then \begin{equation}\label{eq: conj cil}\log L_f(N)\sim \left(d-1\right)N \log N\end{equation} as $N\to\infty$.\end{conj}

\noindent {\bf Convention.} Throughout the rest of the paper in all asymptotic notation the polynomial $f$ is assumed fixed, while the parameter $N$ is taken to infinity.
\\ \\
At present there are no known examples of polynomials of degree $d > 2$ for which the conjecture is known to hold, nor is it known whether such a polynomial even exists. Cilleruelo's argument also shows the predicted upper bound, i.e. if $f \in \mathbb{Z}[X]$ is an irreducible polynomial of degree $d \geq 2$, then
$$\log L_f(N)\lesssim (d-1)N\log N.$$
Maynard and Rudnick \cite{MaRu21} provided a lower bound of the correct order of magnitude:
\begin{equation}\label{eq:sah_lower_bound}\log L_f(N) \gtrsim \frac{1}{d} N \log N.\end{equation}
Sah \cite{Sah20} improved this lower bound to
\begin{equation}\label{eq: sah lower bound L}\log L_f(N)\gtrsim N\log N.\end{equation} He also studied the radical of the least common multiple
$$\ell_f(N)=\mathrm{rad}(L_f(N))=\prod_{p|f(1)f(2)\ldots f(N)\atop{\mathrm{prime}}}p$$ and made the following conjecture:

\begin{conj}\label{conj: sah}Let $f\in\Z[x],\,\deg f=d\ge 2$ be irreducible. Then $$\log\ell_f(N)\sim\log L_f(N)\sim (d-1)N\log N.$$\end{conj}
He proved the following lower bound:
\begin{equation}\label{eq: sah lower bound l}\log \ell_f(N) \gtrsim \frac 2d N \log N.\end{equation}
 We note that conditional on the ABC conjecture one can show that $\log\ell_f(N)\sim\log L_f(N)$, see \cite[Remark 6.2]{EnLa23_2}. 

Rudnick and Zehavi \cite{RuZe20} established an averaged form of (\ref{eq: conj cil}) with $f$ varying in a suitable range which depends on $N$.  Leumi \cite{Leu21_}, as well as the author and Landsberg \cite{EnLa23_2} studied a function field analog of the problem. In the latter work the analog of Conjecture \ref{conj: cilleruelo} was established for a special class of polynomials (which includes many polynomials of degree $d>2$) and it was shown that the analogs of conjectures \ref{conj: cilleruelo} and \ref{conj: sah} are equivalent. Technau \cite{Tec24_1} derived improved lower bounds on $L_f(N),\ell_f(N)$ for even polynomials (i.e. $f(x)=f(-x)$) and a further improved lower bound on $\ell_f(N)$ for the cyclotomic polynomial $f=x^{2^k}+1$.

In the present work we slightly improve the bound (\ref{eq: sah lower bound l}) and give further improved lower bounds on $L_f(N),\ell_f(N)$ for polynomials of special form. Our first main result is the following

\begin{thm}\label{thm: main0} Let $f\in\Z[x],\,\deg f=d\ge 3$ be irreducible. Then
\begin{equation}\label{eq: main0}\log\ell_f(N)\gtrsim \frac{d-1}{V_d}N\log N,\end{equation}
where
\begin{equation}\label{eq: Vd}V_d=\sum_{k=1}^{\min(\lfloor 3d/4\rfloor,d-2)}(d-k)=\left[\begin{array}{ll}2,& d=3,\\
5,&d=4,\\
{ \lfloor\frac{3d}4\rfloor\left(d-\frac{\lfloor\frac{3d}4\rfloor+1}2\right)},&d\ge 5.\end{array}\right.\end{equation}
\end{thm}

The constant in (\ref{eq: main0}) is bigger than the constant in (\ref{eq: sah lower bound l}) for all $d\ge 3$. For large $d$ it is asymptotic to $\frac{32}{15}\cdot\frac 1d$. 

Before stating our next main result, we need the following

\begin{defi}
    A non-empty finite subset $S=\{\alpha_1,\ldots,\alpha_n\}\subset\overline\Q$ is called \emph{potent} if there exist positive integers $a_1,\ldots,a_n$ such that $\sum_{i=1}^na_i\alpha_i\in\Z$.
\end{defi}

\begin{ex}\label{ex: potent}
    Let $h\in\Z[x],\,\deg h\ge 2$ be a polynomial, $a\in\overline\Q$. Then the set of roots of $h(x)-a$ in $\Qb$ is potent.
    Indeed write $h=\sum_{i=0}^sh_ix^i,\,h_s\neq 0$ and let $\alpha_1,\ldots,\alpha_k$ be the distinct roots of $h(x)-a$ with $\alpha_i$ having multiplicity $a_i$. Then by Vieta's formula $\sum_{i=1}^k|h_sa_i|\alpha_i=\pm h_{s-1}\in\Z$.
\end{ex}

Our second main result is the following

\begin{thm}\label{thm: main1} Let $f\in\Z[x],\,\deg f=d\ge 4$ be irreducible and assume that its set of roots $S$ in $\overline\Q$ can be partitioned into (non-empty) potent subsets $S_1,\ldots,S_r$. Then
\begin{enumerate}\item[(i)]
$$\log L_f(N)\gtrsim \frac{d-1}{d-r}N\log N.$$
\item[(ii)]
$$\log\ell_f(N)\gtrsim\frac{d-1}{V_{d,r}}N\log N,$$ where
\begin{equation}\label{eq: Vdr}V_{d,r}=\sum_{k=1}^{\min(\lfloor 3d/4\rfloor,d-2)}(d-\max(k,r))\end{equation}
(note that $V_{d,r}<V_d$ whenever $r>1$).\end{enumerate}\end{thm}

\begin{cor}\label{cor: main1}
Let $f=g\circ h$ be irreducible with $g,h\in\Z[x],\,\deg g=r,\,\deg h=s,\,\deg f=d=rs,\,r,s\ge 2$. Then
$$\log L_f(N)\gtrsim \frac{d-1}{d-r}N\log N.$$

\end{cor}

\begin{proof} Let $\beta_1,\ldots,\beta_r\in\overline\Q$ be the roots of $g$. Taking $R_i=\{\alpha\in\Qb:\,h(\alpha)=\beta_i\}$ ($1\le i\le r$) we obtain the partition required in Theorem \ref{thm: main1} (by Example \ref{ex: potent} applied to $h(x)-\beta_i$ each $R_i$ is potent).
\end{proof}

The special case $h(x)=x^2$ of the above corollary (i.e. when $f(x)=f(-x)$ is even) was independently obtained by Technau \cite{Tec24_1}, with a slightly worse constant in the $\ell_f(N)$ bound. The quadratic case $\deg h=2$ gives the biggest increase in the constant compared to (\ref{eq: sah lower bound L}), namely we obtain the constant $\frac{2(d-1)}d$ instead of $1$.

In the case of a decomposable polynomial $f$ the bound in Theorem 2(ii) can be refined further.

\begin{thm}\label{thm: dec l}
Let $f=g\circ h$ be irreducible with $g,h\in\Z[x],\,\deg g=r,\,\deg h=s,\,\deg f=d=rs,\,r,s\ge 2$. Then
$$\log\ell_f(N)\gtrsim\frac{d-1}{W_{d,r}}N\log N,$$
where
\begin{equation}\label{eq: Wdr} W_{d,r}=\sum_{k=1}^{\min(\lfloor 3d/4\rfloor,d-2)}\max(d-kr,r-\lfloor kr/d\rfloor )\end{equation}
(note that $W_{d,r}<V_{d,r}$ whenever $r,s\ge 2$).
\end{thm}

In the special case $\deg h=\deg g=d^{1/2}$ the constant in Theorem \ref{thm: dec l} grows asymptotically like $\frac{32}{31}\cdot \frac 1{d^{1/2}}$ for large $d$, compared with $\frac{32}{15}\cdot \frac 1d$ in Theorem \ref{thm: main0}.

Our final main result strengthens the bound in Theorem 1 for polynomials with a small Galois group.

\begin{thm}\label{thm: main2}
Let $f\in\Z[x],\deg f=d\ge 3$ be irreducible, $S\subset\overline\Q$ its set of roots, $G_f$ its Galois group, $e=|G_f|$. Suppose that one can partition $S$ into (non-empty) potent subsets $S_1,\ldots,S_r$ where $r$ satisfies $$\dim\langle S\cup\{1\}\rangle_\Q=d+1-r,$$ here $\langle\cdot\rangle_\Q$ denotes the linear span over $\Q$ (note that since the $S_i$ are potent and disjoint, each contributes a new linear dependency on $S\cup\{1\}$ and a priori $\dim\langle S\cup\{1\}\rangle_\Q\le d+1-r$).
Then $$\log\ell_f(N)\gtrsim \frac {d-1}{U_{d,e}}N\log N,$$
where \begin{equation}\label{eq: def U_e}U_{d,e}=\sum_{k=1}^{\min(d-2,\lfloor 3d/4\rfloor)}\min \left(d-k,\bigg\lceil\frac {e}k\bigg\rceil-1\right).\end{equation}
\end{thm}

\begin{rem} We have $U_{d,e}<V_d$ iff $\lceil e/k\rceil-1<d-k$ for some $1\le k\le \min(d-2,\lfloor 3d/4\rfloor)$. In this case the constant in Theorem \ref{thm: main2} is bigger than in (\ref{eq: main0}). In particular this happens whenever $e=d\ge 4$. However there is no improvement if $e>d^2$, so the order of the Galois group of $f$ needs to be at most quadratic in $d$ to obtain an improvement over Theorem \ref{thm: main0}.\end{rem}

\begin{cor}\label{cor: main2} Let $f\in\Z[x],\,\deg f=d\ge 3$ be an irreducible polynomial with roots $\alpha_1,\ldots,\alpha_d\in\Qb$ such that
\begin{enumerate}
    \item [(a)] $\Q(\alpha_1)=\Q(\alpha_1,\ldots,\alpha_d)$.
    \item [(b)] $\dim\langle 1,\alpha_1,\ldots,\alpha_d\rangle_{\Q}=d$.
\end{enumerate}
Then $$\log \ell_f(N)\gtrsim \frac {d-1}{U_d}N\log N,$$ where $$U_d=U_{d,d}=\sum_{k=1}^{\min(d-2,\lfloor 3d/4\rfloor)}(\lceil d/k\rceil-1).$$
\end{cor}

\begin{proof} Condition $(a)$ is equivalent to $e=|G_f|=d$, the set $\{\alpha_1,\ldots,\alpha_n\}$ is potent by Example \ref{ex: potent}. We may apply Theorem \ref{thm: main2} with the partition $S=S_1$ and obtain the assertion.
\end{proof}

\begin{rem} Condition (b) in Corollary \ref{cor: main2} typically holds, it says that there are no linear relations between the roots of $f$ and 1 other than the one imposed by Vieta's formula. Condition (a) is much more special and is equivalent to $f$ being the minimal polynomial (normalized to have coprime integer coefficients) of a primitive element in a Galois extension of $\Q$. To produce examples of polynomials $f$ satisfying (a) and (b) one can take an arbitrary Galois extension $E/\Q$, pick a ``random" element $\alpha\in E$ and typically its minimal polynomial $f$ will satisfy (a) and (b).
\end{rem}

\begin{rem} The first few values of $U_d$ are $U_3=2, U_4=4, U_5=7, U_6=9, U_7=13, U_8=15, U_9=18, U_{10}=21$. For large $d$ we have $U_d\sim d\log d$, so the constant in Corollary \ref{cor: main2} grows like $1/\log d$ as $d\to\infty$, compared with $32/15d$ in Theorem \ref{thm: main0}. The actual predicted constant is $d-1$.
\end{rem}

\begin{cor}\label{cor: cyclotomic}
    Let $\Phi_m(x)$ be the $m$-th cyclotomic polynomial, with $m=p^kn$, $p$ prime, $n$ squarefree. Denote $d=\deg\Phi_m=\varphi(m)$ and assume $d\ge 3$. Then
    $$\log\ell_{\Phi_m}(N)\gtrsim\frac {d-1}{U_d}N\log N.$$
\end{cor}

\begin{proof} Denote by $S$ the set of roots of $\Phi_m$, i.e. the primitive $m$-th roots of unity and let $\zeta\in S$ be one such root. Note that $\Q(\zeta)=\Q(S)$ and $|\Gal(\Phi_m/\Q)|=d=\varphi(m)$. 

If $m$ is itself squarefree then by \cite[p. 57]{Joh85} $S$ is a basis of $\Q(\zeta)$ and therefore $\dim\langle S\cup\{1\}\rangle_\Q=d$ (since $S$ is potent $S\cup\{1\}$ is linearly dependent. In fact $\sum_{\alpha\in S}\alpha=\mu(m)$, where $\mu$ is the M\"obius function). Hence Corollary \ref{cor: main2} applies, giving the assertion.

Now assume that $m$ is not squarefree, so we may assume $m=p^kn$, $k\ge 2$, $(p,n)=1$.
Denote 
$$\{a_1,\ldots,a_{d/p}\}=\{1\le a<m/p: (a,m/p)=1\}$$ and set
$$S_i=\{\zeta^{a_i},\zeta^{a_i+m/p},\zeta^{a_i+2m/p}+\ldots\zeta^{a_i+(p-1)m/p}\},\quad 1\le i\le d/p.$$
Then $S_1,\ldots,S_r$ with $r=d/p$ is a partition of $S$ and each $S_i$ is potent since $\sum_{\alpha\in S_i}\alpha=0$. Finally by \cite[p. 57]{Joh85} we have $\dim\langle S\rangle_\Q=(p-1)d/p$ and $1\not\in\langle S\rangle_\Q$ because $\mathrm{tr}_{\Q(\zeta)/\Q}(\alpha)=\sum_{\beta\in S}\beta=\sum_{i=1}^{d/p}\sum_{\beta\in S_i}\beta=0$ for any $\alpha\in S$, while $\mathrm{tr}_{\Q(\zeta)/\Q}(1)\neq 0$. Hence $\dim\langle S\cup\{1\}\rangle_\Q=(p-1)d/p+1=d-r+1$, so Theorem \ref{thm: main2} applies with $e=d$ and we obtain the assertion of the corollary.
\end{proof}

\begin{rem} In the special case $m=2^k$ the sharper bound $\log \ell_{\Phi_m}(N)\gtrsim \frac{2^{k-1}-1}{(k-1)2^{k-2}}N\log N$ was obtained in \cite{Tec24_1} by a different method.\end{rem}
\begin{rem} In the setup of Theorem \ref{thm: main2}, if additionally the set of roots of $f$ admits a partition into $u>1$ potent subsets, then one could combine the proofs of theorems \ref{thm: main1}, \ref{thm: main2} to obtain a slightly sharper constant with $U_{d,e}$ replaced by
$$\sum_{k=1}^{\min(d-2,\lfloor 3d/4\rfloor)}\min \left(d-\max(k,u),\bigg\lceil\frac {e}k\bigg\rceil-1\right).$$
This applies in particular to the cyclotomic polynomial $\Phi_m$ where $m$ is composite and satisfies the condition of Corollary \ref{cor: cyclotomic}, but the improvement is modest.
\end{rem}
The rest of the paper is dedicated to the proof of theorems \ref{thm: main0}, \ref{thm: main1}, \ref{thm: dec l}, \ref{thm: main2}.
\\ \\
{\bf Acknowledgment.} The author would like to thank Ze\'ev Rudnick and the anonymous referee of this manuscript for helpful comments on the exposition. The author was partially supported by Israel Science Foundation grant no. 2507/19.

\section{General strategy and proof of Theorem \ref{thm: main0}}

The usual approach to estimating $L_f(N)$ and $\ell_f(N)$ is to consider their prime factorization and separately estimate the contribution of small primes $p=O(N)$ and large primes (not falling in this range). The contribution of small primes can be estimated precisely using the Chebotarev density theorem (in fact even weaker results on the distribution of prime ideals suffice), but understanding the contribution of large primes is a significant challenge. The approach to lower bounds taken in \cite{MaRu21, Sah20} as well as here is to bound the number of values $f(1),\ldots,f(N)$ that a given large prime can divide, with or without counting multiplicity. Our new results come from giving improved bounds of this type for polynomials of special form. An additional ingredient we use for the lower bounds on $\ell_f(N)$ are results of Hooley \cite{Hoo67} and Browning \cite{Bro11} from their work on power-free values of polynomials.

\begin{prop}\label{prop: main1} Let $f\in\Z[x],\deg f=d\ge 2$ be an irreducible polynomial. Suppose there exist constants $D=D(f)>0, \delta=\delta(f)>0$ such that for any sufficiently large $N\ge N_0(f)$ and any prime $p>DN$, $p$ divides at most $\delta$ of the values $f(1),f(2),\ldots,f(N)$. Then
$$L_f(N)\gtrsim \frac{d-1}{\delta}N\log N.$$
\end{prop}

\begin{proof} 
This is proved in \cite[Proof of Theorem 1.3]{Sah20} for the special case $ D=d|f_d|+1,\,\delta=d-1$, but the proof works essentially verbatim for any $D,\delta>0$.\end{proof}

For the analogous proposition for $\ell_f(N)$ we need to recall the results of Hooley and Browning which were key steps in their work on power-free values of polynomials.

\begin{prop}\label{prop: k-free} Let $f\in\Z[x],\deg d\ge 3$ be an irreducible polynomial.
 The number of values among $f(1),\ldots,f(N)$ which are divisible by $p^{\min(d-1,\lfloor 3d/4\rfloor+1)}$ for some prime $p>N$ is $o(N)$.\end{prop}

\begin{proof} It follows from \cite[Equation (15)]{Hoo67} that the number of values among $f(1),\ldots,f(N)$ divisible by $p^{d-1}$ for a prime $p>N$ is $o(N)$. The same statement with $p^{\lfloor 3d/4\rfloor+1}$  in place of $p^{d-1}$ follows from the bound established in \cite[\S 3]{Bro11} on the quantity $E(\xi)$ defined in \cite[Equation (9)]{Bro11}.
\end{proof}

\begin{prop} Let $f\in\Z[x],\,\deg f=d\ge 3$ be an irreducible polynomial. \label{prop: main2} Suppose there exist constants $D=D(f)>0, \delta_k=\delta_k(f)\ge 0\,(1\le k\le d)$ such that for any sufficiently large $N\ge N_0(f)$ and any prime $p>DN$, $p^k$ divides at most $\delta_k$ of the values $f(1),\ldots,f(N)$. Then
$$\ell_f(N)\gtrsim \frac{d-1}{V}N\log N,$$
where \begin{equation}\label{eq: V}V=\sum_{k=1}^{\min(d-2,\lfloor 3d/4\rfloor)}\delta_k.\end{equation}
\end{prop}

\begin{proof}
 The same assertion but with the constant $\frac{d-1}{\sum_{k=1}^d\delta_k}$ is essentially established in \cite[Proof of Theorem 1.4]{Sah20}. The main new ingredient introduced here is Proposition \ref{prop: k-free}.
We denote by $\sum_p,\prod_p$ sums and products over primes respectively. By $v_p(n)$ we denote the multiplicity of $p$ in the factorization of $n$. Assume WLOG that $D\ge 1$ and denote $$Q(N)=\prod_{n=1}^Nf(n),\quad Q_1(N)=\prod_{p\atop{p\le DN}}p^{v_p(f(n))},\quad Q_L(N)=\prod_{p\atop{p> DN}}p^{v_p(f(n))},\quad \ell_L(N)=\prod_{p|Q(N)\atop{p> DN}}p.$$
We have $Q(N)=Q_1(N)Q_L(N)$ and 
$\ell_f(N)\ge\ell_L(N)$. In \cite{Sah20} 
the quantity $Q_1(N)$ is further 
decomposed as $Q_1(N)=Q_S(N)Q_{LI}(N)$ 
(products over small and linear-sized 
primes) and it is shown that $\log Q_S(N)\sim 
N\log N$ \cite[Proposition 2.1]{Sah20} and 
$\log Q_{LI}(N)=O(N)$ \cite[Proposition 
3.1]{Sah20}. Since $\log Q(N)\sim dN\log 
N$ we obtain
$$\log Q_L(N)=\log\frac{Q(N)}{Q_S(N)Q_{LI}(N)}\sim (d-1)N\log N.$$

Now fix $0<\epsilon<1$. By Proposition \ref{prop: k-free}, for $N$ sufficiently large at most $\epsilon N$ of the values $f(1),\ldots,f(N)$ are divisible by the $m$-th power of any prime $p>DN$, where $m=\min(d-1,\lfloor 3d/4\rfloor+1)$, while at most $\delta_k$ are divisible by the $k$-th power of such a prime for each $1\le k\le d$.
Thus $v_p(f(n))\le V$ (where $V$ is given by (\ref{eq: V})) for all but at most $\epsilon N$ of $1\le n\le N$ and $v_p(f(n))\le\delta=\sum_{k=1}^d\delta_k$ for all of them (taking $D$ large enough we may assume WLOG that $p^{d+1}>|f(N)|$ for any $p>DN$, and then $p^{d+1}\nmid f(n)$ for any $1\le n\le N$). Hence (using that $\log p<(d+1)\log N$ if $p|Q(N)$ and $N$ is sufficiently large)
\begin{multline*}(d-1)N\log N\sim\log Q_L(N)=\sum_{p|Q(N)\atop{p>DN}}v_p(f(n))\log p\le \epsilon N\delta(d+1)\log N +\sum_{p|Q(N)\atop{p>DN}}V\log p\\=
\epsilon\delta (d+1)N\log N+V\log\ell_L(N).\end{multline*}
Taking $\epsilon\to 0$ shows that
$$\log \ell_f(N)\ge\log \ell_L(N)\gtrsim\frac{d-1}{V}N\log N,$$ proving the assertion.
\end{proof}

The next lemma is a generalization of \cite[Lemma 4.1]{Sah20}.

\begin{lem}\label{lem: sah} Let $f=\sum_{i=0}^df_ix^i\in\Z[x], \deg f=d\ge 2$ be a fixed polynomial. Denote $D=d|f_d|+1$. There exists a constant $M=M(f)\in\N$, such that for any $N\ge M$ and integers $1\le k\le d-1$, $a\in\Z$, $q>(DN)^k$, there do not exist integers $M\le n_1<n_2<\ldots<n_{d-k+1}\le N$ such that:
\begin{enumerate}
    \item[(a)] $f(n_i)\equiv a\pmod q$ for all $1\le i\le d-k+1$,
    \item[(b)] $(q, n_i-n_j)=1$ for all $1\le i<j\le d-k+1$.
\end{enumerate}
\end{lem}

\begin{proof}

Assume that $1\le n_1<n_2<\ldots<n_{d-k+1}\le N$ satisfy properties (a),(b) from the statement of the lemma. Denote \begin{equation}\label{eq: lagrange}A=\sum_{i=1}^{d-k+1}\frac{f(n_i)-a}{\displaystyle\prod_{1\le j\le d-k+1\atop{j\neq i}}(n_i-n_j)}=\sum_{i=1}^{d-k+1}\frac{f(n_i)}{\displaystyle\prod_{1\le j\le d-k+1\atop{j\neq i}}(n_i-n_j)},\end{equation} 
with the last equality stemming from the well-known identity
$$\sum_{i=1}^{d-k+1}\frac{1}{\displaystyle\prod_{1\le j\le d-k+1\atop{j\neq i}}(n_i-n_j)}=0,$$
which can be obtained by considering the Lagrange interpolation $L(x)$ of the constant 1 at the points $n_1,\ldots,n_{d-k+1}$ and taking the limit $\lim_{x\to\infty}L(x)/x^{d-k}$ (note that $d-k+1\ge 2$ since $k\le d-1$).

It is shown in \cite[Proof of Lemma 4.1]{Sah20} that $A$ (as given by the RHS of (\ref{eq: lagrange})) is an integer and $|A|\le (DN)^k$ for $N$ sufficiently large, hence $|A|<q$. However conditions (a),(b) imply that $q$ divides the numerator and is coprime to the denominator of each summand in the middle expression in (\ref{eq: lagrange}), hence $q|A$. Consequently $A=0$.

It is shown in \cite[Proof of Lemma 4.1]{Sah20} that $A=0$ can only happen if $n_1,\ldots,n_{d-k+1}< M$ for some constant $M=M(f)$. This completes the proof.
\end{proof}

The following corollary to the above lemma essentially appears in \cite{Sah20}.

\begin{cor}\label{cor: sah} Let $f=\sum_{i=0}^df_ix^i\in\Z[x],\,\deg f=d\ge 2$ be a fixed polynomial without (rational) integer roots, $D=d|f_d|+1$. Then for sufficiently large $N\ge N_0(f)$, a prime $p>DN$ and $1\le k\le d-1$, at most $d-k$ of the values $f(1),\ldots,f(N)$ are divisible by $p^k$.\end{cor}

\begin{proof} Let $M=M(f)$ be the constant provided by Lemma \ref{lem: sah}. Assume that there exist $1\le n_1<\ldots<n_{d-k+1}<N$ such that $p^k|f(n_i),\,1\le i\le d-k+1$. Then for $a=0,q=p^k>(DN)^k$ conditions (a),(b) from Lemma \ref{lem: sah} hold (condition (b) holds because $0<|n_j-n_i|<N<p\,\,(i\neq j)$, so each $n_j-n_i$ is coprime with $q=p^k$). Since $f$ has no integer roots, for sufficiently large $N$ we have $0<|f(n)|<N<p$ for any $1\le n<M$, so we must have $n_1\ge M$. Hence $n_1,\ldots,n_{d-k+1}$ are a counterexample to the assertion of Lemma \ref{lem: sah} and cannot exist for $N$ sufficiently large.\end{proof}

\begin{proof}[Proof of Theorem \ref{thm: main0}] By Corollary \ref{cor: sah}, for sufficiently large $N$, a prime $p>DN$ and $1\le k\le d-1$, $p^k$ divides at most $d-k$ of the values $f(1),\ldots,f(N)$.
Hence we can apply Proposition \ref{prop: main2} with $\delta_k=d-k$ and obtain 
$\log \ell_f(N)\gtrsim \frac{d-1}VN\log N$ with $$V=\sum_{k=1}^{\min(d-2,\lfloor 3d/4\rfloor)}(d-k)=V_d.$$
\end{proof}

\section{Proof of Theorem \ref{thm: main1}}

We will prove Theorem \ref{thm: main1}(i) by showing that the condition of Proposition \ref{prop: main1} holds for a polynomial and parameter $r$ as in the statement of the theorem, with $\delta=d-r$. Similarly, Theorem \ref{thm: main1}(ii) will follow by applying Proposition \ref{prop: main2} with $\delta_k=d-\max(k,r)$. First we need an algebraic lemma, which will also be used in the proofs of theorems \ref{thm: dec l} and \ref{thm: main2}.

\begin{lem}\label{lem: hom} Let $f=\sum_{i=0}^df_ix^i\in\Z[x],\,f_d\neq 0$ be a squarefree polynomial with roots $\alpha_1,\ldots,\alpha_d$ in $\Qb$, $K=\Q(\alpha_1,\ldots,\alpha_d)$ the splitting field of $f$, $p$ a prime number such that $p\nmid f_d\cdot\disc(f)$, $k\ge 1$. Then there exists a ring homomorphism $\phi:\Z[\alpha_1,\ldots,\alpha_d]\to A$, where $A$ is an extension ring of $\Z/p^k\Z$, such that 
\begin{enumerate}
    \item[(i)] $\phi$ induces a bijection between $\{\alpha_1,\ldots,\alpha_d\}$ and the set of roots of $f$ in $A$.
    \item[(ii)] If $\gamma\in\Z[\alpha_1,\ldots,\alpha_d]$ is integral over $\Z$ and $\phi(\gamma)=0$ then $p^k|N_{K/\Q}(\gamma)$.
\end{enumerate}
\end{lem}

\begin{proof} Let $\mathcal O$ be the ring of integers of $K$ and $B=\mathcal O_{f_d}$ (localization at $f_d$). Let $\mathfrak p$ be a prime ideal of $\mathcal O$ lying over $p$. Then since $p\nmid f_d$, the inclusion $\mathcal O\hookrightarrow B$ induces an isomorphism $A:=\mathcal O/\mathfrak p^k\isom B/\mathfrak p^kB$. Since $p\nmid f_d\cdot\disc(f)$, $p$ is unramified in the extension $K/\Q$ and therefore $\mathfrak p^k\cap\Z=p^k$ and $A$ is an extension ring of $\Z/p^k\Z$.

Take $\phi$ to be the composition of the inclusion $\Z[\alpha_1,\ldots,\alpha_d]\hookrightarrow B$, the quotient map $B\to B/\mathfrak p^kB$, and the isomorphism $B/\mathfrak p^kB\isom A$. Note that $\phi$ takes roots of $f$ to roots of $f$ in $A$ and further since $p\nmid f_d\cdot\disc(f)$ we have $\alpha_i-\alpha_j\not\in\mathfrak p B\supset\mathfrak p^kB$ if $i\neq j$, so $\phi$ is injective on $\{\alpha_1,\ldots,\alpha_d\}$. Finally note that since $p\nmid f_d\cdot\disc(f)$, the reduction of $f$ modulo $\mathfrak p$ has at most $d$ roots in $\mathcal O/\mathfrak p$ and each of them has exactly one lifting to $A=\mathcal O/\mathfrak p^k$ by Hensel's lemma. Since $\phi(\alpha_1),\ldots,\phi(\alpha_d)$ are $d$ distinct roots of $f$ modulo $\mathfrak p^k$, these must be all of the roots, giving (i).

For (ii) note that $\phi(\gamma)=0$ is equivalent to $\mathfrak p^k|\gamma$. Taking norms and using $p|N_{K/\Q}(\mathfrak p)$, we obtain $p^k|N_{K/\Q}(\gamma)$.

\end{proof}

The next lemma is the key result for the proof of Theorem \ref{thm: main1}.

\begin{lem}\label{lem: main1} Let $f=\sum_{i=0}^df_ix^i\in\Z[x],\,\deg f=d$ be squarefree and without (rational) integer roots, $S$ the set of roots of $f$ in $\Qb$, $S_1,\ldots,S_r$ a partition of $S$ with $S_i$ potent. There exists a constant $D=D(f)>0$ such that for any sufficiently large $N\ge N_0(f)$ and any prime $p>DN$, $p$ divides at most $d-r$ of the values $f(1),f(2),\ldots,f(N)$.
\end{lem}

\begin{proof}
Denote $S_i=\{\alpha_{ij}:\,1\le j\le s_i\}$. Since $S_i$ are assumed potent, for each $1\le i\le r$ there exist $0< b_{ij}\in\Z,\,1\le j\le s_i$ such that \begin{equation}\label{eq: comb b alpha}\sum_{j=1}^{s_i}b_{ij}\alpha_{ij}={v_{i}}\quad (1\le i\le r)\end{equation} for some $v_{i}\in\Z$.
We take $D=\sum_{i,j}b_{ij}+\sum_i|v_i|$ and claim that the assertion of the lemma holds with the constant $D$.

Let $p>DN$ be a prime. For sufficiently large $N$ we have $p\nmid f_d\cdot\disc(f)$. Lemma \ref{lem: hom} applies and there exists an extension ring $A\supset\F_p$ and a ring homomorphism $\phi:\Z[S]\to A$ which induces a bijection from $S$ to the set of roots of $f$ in $A$. For any $\alpha\in\Z[S]$ (resp. $R\subset S$) we denote $\bar\alpha=\phi(\alpha)$ (resp. $\bar R=\phi(R)$, the direct image of $R$ under $\phi$). Since $\phi$ is injective on $S$, $\{\bar S_1,\ldots,\bar S_r\}$ is a partition of $\bar S$ and $|\bar S_i|=s_i,|\bar S|=d$.

Assume by way of contradiction that $p$ divides at least $d-r+1$ of the values $f(1),f(2),\ldots,f(N)$. In other words the polynomial $f$ has at least $d-r+1$ distinct roots modulo $p$ with integer representatives in the range $\{1,2,\ldots,N\}$ (note that $p>N$), denote the set of these roots by $T$. Since $\phi$ is a homomorphism we have $T\subset \bar S$. Since $|T|\ge d-r+1$, there must be some $1\le i\le r$ such that $\bar S_i\subset T$, otherwise (if $T$ misses some element from each $\bar S_i,\,1\le i\le r$) we would have $|T|\le d-r$.

Applying $\phi$ to (\ref{eq: comb b alpha}) we have 
$$\sum_{j=1}^{s_i}b_{ij}\bar\alpha_{ij}=\bar v_i$$
and since $\bar\alpha_{ij}\in\bar S_i\subset T$ we can write $\bar\alpha_{ij}=\bar n_{ij}$ for some integers $1\le n_{ij}\le N$, so we have
\begin{equation}\label{eq: div by p}\sum_{j=1}^{s_i}b_{ij}n_{ij}-v_i\equiv 0\pmod p.\end{equation}

Note that $f(\bar n_{ij})=f(\bar\alpha_{ij})=0$, so $f(n_{ij})\equiv 0\pmod p$, however $f(n_{ij})\neq 0$ since $f$ is assumed to be without integer roots. Hence for sufficiently large $N\ge\sum_{j=0}^d|f_jv_i^j|$ and $p>N$ we must have $n_{ij}>|v_i|$ (otherwise $0<f(n_{ij})<p$, but $p|f(n_{ij})$, a contradiction). Consequently the LHS of (\ref{eq: div by p}) is positive but also $\le \sum_{i,j}b_{ij}N+\sum_i|v_i|N=DN<p$, a contradiction to (\ref{eq: div by p}). 

\end{proof}

\begin{proof}[Proof of Theorem \ref{thm: main1}] Assume the setup of Theorem \ref{thm: main1}.

{\bf (i).}
 By Lemma \ref{lem: main1} we may apply Proposition \ref{prop: main1} with $\delta=d-r$ and conclude that
$$\log L_f(N)\gtrsim \frac{d-1}{d-r}N\log N.$$

{\bf (ii).} By Lemma \ref{lem: main1} combined with Corollary \ref{cor: sah} we may apply Proposition \ref{prop: main2} with $\delta_k=d-\max(k,r)$ and conclude that
$$\log \ell_f(N)\gtrsim \frac{d-1}{V_{d,r}}N\log N$$ with $V_{d,r}$ given by (\ref{eq: Vdr}).\end{proof}

\section{Proof of Theorem \ref{thm: dec l}}

We will prove Theorem \ref{thm: dec l} by showing that the condition of Proposition \ref{prop: main2} holds for a polynomial as in the statement of the theorem and 
\begin{equation}\label{eq: delta for dec l}\delta_k=\max(d-kr,r-\lfloor kr/d\rfloor)=\left[\begin{array}{ll}d-kr,&1\le k\le   d/r-1,\\ r-\lfloor kr/d\rfloor,& d/r\le k\le d-1.\,\end{array}\right.\end{equation}

\begin{lem}\label{lem: dec l}
Let $f=g\circ h,\,g,h\in\Z[x]$ be a decomposable squarefree polynomial without (rational) integer roots, $\deg g=r\ge 2,\,\deg h\ge 2$, $\deg f=d$. There exists a constant $D=D(f)>0$ such that for all sufficiently large $N\ge N_0(f)$, prime $p>DN$, and $1\le k\le d-1$, there are at most $\delta_k$ values among $f(1),\ldots,f(N)$ divisible by $p^k$, where $\delta_k$ is given by (\ref{eq: delta for dec l}).
\end{lem}

\begin{proof}

Denote 
$s=\deg h=\frac dr,\,f=\sum_{i=0}^df_ix^i,\, g=\sum_{i=0}^rg_ix^i,\,h=\sum_{i=0}^sh_ix^i.$
Let $\alpha_1,\ldots,\alpha_d\in\Qb$ be the roots of $f$ and $\beta_1,\ldots,\beta_r\in\Qb$ the roots of $g$. Since $f=g\circ h$ is squarefree, the sets $S_i=h^{-1}(\beta_i),\,1\le i\le r$ form a partition of $\{\alpha_1,\ldots,\alpha_d\}$ with $|S_i|=s$.

Let $1\le k\le d-1$. By Lemma \ref{lem: hom}, for a sufficiently large prime $p$ there exists an extension ring $A\supset\Z/p^k\Z$ and a ring homomorphism $\phi:\Z[\alpha_1,\ldots,\alpha_d]\to A$ which induces a bijection from $S=\{\alpha_1,\ldots,\alpha_d\}$ to the set of roots of $f$ in $A$. For any $\alpha\in\Z[\alpha_1,\ldots,\alpha_d]$ denote $\bar\alpha=\phi(\alpha)$ and for a set $T\subset\Qb$ denote $\bar T=\phi(S)\subset A$. Since $\phi$ is injective on $S$, the sets $\bar S_1,\ldots,\bar S_r$ form a partition of $\bar S$.

We pick $$D=\max\left((r|g_r|+1)(|h_s|+1),s|h_s|+1\right)$$ and let $M\in\N$ be a constant for which the assertion of Lemma \ref{lem: sah} holds for both the polynomials $g$ and $h$ (in place of $f$) with the respective constants $D_g=r|g_r|+1,D_h=s|h_s|+1$.
We separate the rest of the proof into two cases corresponding to the ranges $k<s$ and $k\ge s$ appearing in (\ref{eq: delta for dec l}).

\case{$k\le s-1$} By the assertion of Lemma \ref{lem: sah} applied to the polynomial $h$ with $q=p^k$, for any $N\ge M$, a prime $p>DN$ and any $a\in\Z$, at most $s-k$ of the values $h(M)-a,h(M+1)-a,\ldots,h(N)-a$ are divisible by $p^k$.

We claim that for  sufficiently large $N$ and a prime $p>DN$, at most $\delta_k=d-kr$ of the values $f(1),\ldots,f(N)$ are divisible by $p^k$, or equivalently (since $p^k>p>DN>N$) at most $d-kr$ elements of $\bar S$ fall in $\Z/p^k\Z$ and have representatives in the interval $\{1,\ldots,N\}$.

Assume by way of contradiction that there are $d-kr+1=r(s-k)+1$ such elements. 
Then by the pigeonhole principle $s-k+1$ of them fall in some $\bar S_i$, say $\bar S_1=h^{-1}(\beta_1)$. Assume WLOG that these are $\bar\alpha_1,\ldots,\bar\alpha_{s-k+1}$ and that $\bar\alpha_i=\bar n_i,\,1\le n_1<n_2<\ldots<n_{s-k+1}\le N$. 
For sufficiently large $N$ we cannot have $n_1<M$. Indeed since $f(\bar n_1)=f(\bar\alpha_1)=0$ and therefore $f(n_1)\equiv 0\pmod {p^k}$. The polynomial $f$ is assumed to be without integer roots, so $0<|f(n_1)|<p^k$ for $1\le n_1<M$, large enough $N$ and $p>DN\ge D_hN$. This contradicts $p^k|f(n_1)$, establishing the claim $n_1\ge M$.

Therefore we may assume $M\le n_1<n_2<\ldots<n_{s-k+1}\le N$. 
Denote $a=h(n_1)$. Then $\bar a=h(\bar \alpha_1)=\bar\beta_1=h(\bar \alpha_i)=h(\bar n_i),\,1\le i\le s-k+1$ and consequently $p^k|h(n_i)-a$ for all $1\le i\le s-k+1$ (here we used $\bar\alpha_i\in\bar S_1=h^{-1}(\beta_1)$). This contradicts the assertion of Lemma \ref{lem: sah} for the polynomial $h$ (with $q=p^k$; note that $(p^k,n_i-n_j)=1\,(i\neq j)$ because $0<|n_i-n_j|<N<DN<p$) assumed above, establishing our claim and the assertion of the present lemma in the case $k\le s-1$. 

\case{$k\ge s$} We claim that for large enough $N$ and prime $p>DN$, at most $\delta_k=r-\lfloor k/s\rfloor$ of the values $f(1),\ldots,f(N)$ are divisible by $p^k$. 
Assume that $1\le n_1<n_2<\ldots<n_{r-\lfloor k/s\rfloor+1}\le N$ are such that $p^k|f(n_i),\,1\le i\le r-\lfloor k/s\rfloor+1$. 

We may assume WLOG that $h_s>0$, otherwise replace $h$ with $-h$ and $g(x)$ with $g(-x)$ (this leaves $f$ unchanged).

Denote $m_i=h(n_i),\,1\le i\le r-\lfloor k/s\rfloor+1$. First we claim that for sufficiently large $N$ the $m_i$ are all distinct and $m_i\ge M$. Indeed since $p|f(n_i)$, $p$ is large and $f$ has no integer roots, each $n_i$ must be large. Since we assumed $h_s>0$ we have $h(n_{r-\lfloor k/s\rfloor+1})>\ldots>h(n_2)>h(n_1)>M$.
Now for sufficiently large $N$ we have $$m_i=h(n_i)\le (h_s+1)N^s<(DN)^s<p^s\le p^k$$ and so $\bar m_1,\ldots,\bar m_{r-\lfloor k/s\rfloor+1}$ are distinct. Consequently $\bar m_i=f(\bar n_i)=f(\bar\alpha_i)=\bar\beta_i$ for a suitable numbering of the roots $\beta_i$ of $g$. 

For $p>|g_r\cdot\disc(g)|$ large the $\bar\beta_i$ are distinct modulo $pA$, so the $m_i$ are distinct modulo $p$. We may apply the assertion of Lemma \ref{lem: sah} to the polynomial $g$ with $k'=\lfloor k/s\rfloor\le r-1, N'=(h_s+1)N^s$ in place of $k,N$, $a=0$ and 
$$q=p^k> (DN)^k\ge\left((h_s+1)(r|g_r|+1)N\right)^{k}\ge ((r|g_r|+1)N')^{\lfloor k/s\rfloor}=(D_gN')^{\lfloor k/s\rfloor}.$$
Since the $M\le m_i\le N$ are $r-\lfloor k/s\rfloor$ distinct integers with $(q,m_i-m_j)=1\,\,(i\neq j)$ and $q|f(m_i)$ we obtain a contradiction to the assertion of Lemma \ref{lem: sah}. This concludes the proof in the case $k\ge s$.

\end{proof}

\begin{proof}[Proof of Theorem \ref{thm: dec l}]
By Lemma \ref{lem: dec l} we may apply Proposition \ref{prop: main2} with $\delta_k$ given by (\ref{eq: delta for dec l}) and conclude that
$$\log \ell_f(N)\gtrsim \frac{d-1}{W_{d,r}}N\log N$$ with $W_{d,r}$ given by (\ref{eq: Wdr}).
\end{proof}

\section{Proof of Theorem \ref{thm: main2}}

We will prove Theorem \ref{thm: main2} by showing that the condition of Proposition \ref{prop: main2} holds for a polynomial as in the statement of the theorem and $\delta_k=\min(d-k,\lceil e/k\rceil-1)$, where $e=|\Gal(f/\Q)|$.

\begin{lem}\label{lem: int comb} Let $N,n_1,n_2,\ldots,n_t$ be natural numbers with $n_i\le N$. There exist integers $a_1,\ldots,a_t$, not all zero, such that $|a_i|\le (tN)^{1/t}+1$ and $|a_1n_1+\ldots+a_tn_t|\le (tN)^{1/t}$.\end{lem}

\begin{proof} Denote $C=\lceil(tN)^{1/t}\rceil$. For each tuple $\mathbf b=(b_1,\ldots,b_t)\in\{0,1,\ldots,C\}^t$ consider the value $\psi(\mathbf b)=\sum_{i=1}^tb_in_i$. We have $0\le\psi(\mathbf b)\le tCN$, hence among the $(C+1)^t$ possible tuples there must be two $\mathbf b=(b_1,\ldots,b_t)\neq\mathbf b'=(b_1',\ldots,b_t')$ with 
$|\psi(\mathbf b)-\psi(\mathbf b')|<\frac{tCN}{(C+1)^t-1}<\frac{tN}{C^{t-1}}<(tN)^{1/t}$. Taking $a_i=b_i-b_i'$ and noting that $\psi(\mathbf b)-\psi(\mathbf b')=\sum_{i=1}^ta_in_i$ gives the desired result.\end{proof}

We now state and prove the key lemma which will allow us to prove Theorem \ref{thm: main2}.

\begin{lem}\label{lem: pre lem main2} Let $f\in\Z[x],\,\deg f=d$ be squarefree and without (rational) integer roots, $S$ the set of roots of $f$ in $\Qb$, $S_1,\ldots,S_r$ a partition of $S$ with $S_i$ potent and $r$ satisfying $\dim\langle S\cup\{1\}\rangle_\Q=d+1-r$. Denote $G_f=\Gal(f/\Q)$.
Then there exists a constant $D=D(f)>0$ such that for any sufficiently large $N> N_0(f)$, a prime $p>DN$ and $k\ge 1$, $p^k$ divides at most
$\big\lceil\frac{|G_f|}k\big\rceil-1$ of the values $f(1),f(2),\ldots,f(N)$.
\end{lem}

\begin{proof} 
Assume that $p^k|f(n_1),\ldots,f(n_t)$ for some $1\le n_i\le N$. We need to show that $t<\frac{|G_f|}k$ whenever $p>DN$ for a suitable constant $D$ and $N$ sufficiently large. Denote by $K/\Q$ the splitting field of $f$. By Lemma \ref{lem: hom}, for $p$ sufficiently large there exists an extension ring $A\supset \Z/p^k\Z$ and a ring homomorphism $\phi:\Z[S]\to A$ with the following properties:
\begin{enumerate}
    \item[(i)] $\phi$ induces a bijection between $S$ and the roots of $f$ in $A$.
    \item[(ii)] For any integral $\gamma\in\Z[S]$ with $\phi(\gamma)=0$ we have $p^k|N_{K/\Q}(\gamma)$.
\end{enumerate}
For any $\alpha\in\Z[S]$ we denote $\bar\alpha=\phi(\alpha)$.
By (i) and the assumption $p^k|f(n_i)$, there exist $\alpha_1,\ldots,\alpha_t\in S$ such that $\bar\alpha_i=\bar n_i$. If we take $D\ge 1$ so that $p>N$, the elements $\alpha_1,\ldots,\alpha_t$ are distinct (since the $\bar n_i$ are). 

 By Lemma \ref{lem: int comb} there exist integers $a_1,\ldots,a_t,\,|a_i|\le (tN)^{1/t}+1$, not all zero, such that $u=\sum_{i=1}^ta_in_i$ satisfies $|u|\le (tN)^{1/t}$. 
Denote by $f_d$ the leading coefficient of $f$ and consider the element
\begin{equation}\label{eq: gamma}\gamma=f_d\sum_{i=1}^ta_i(\alpha_i-n_i)=f_d\sum_{i=1}^t a_i\alpha_i-f_du\in\Z[S].\end{equation}
Since $\bar\alpha_i=\bar n_i$ we have $\bar\gamma=0$. Also $\gamma$ is integral over $\Z$ (because each $f_d\alpha_i$ is) so by property (ii) of $\phi$ we have $p^k|N_{K/\Q}(\gamma).$
 
 We view $\Qb$ as a subset of $\C$, the latter equipped with the usual absolute value. Denote $C=\max_{\alpha\in S}|\alpha|$.
 From (\ref{eq: gamma}) we see that any conjugate $\gamma'$ of $\gamma$ satisfies
 $$|\gamma'|\le f_d\sum_{i=1}^t|a_i|C+f_d|u|\le f_d\cdot\left(((tN)^{1/t}+1)C+(tN)^{1/t}\right)$$ and thus
$$|N_{K/\Q}(\gamma)|\le f_d^{|G_f|}(((tN)^{1/t}+1)C+(tN)^{1/t})^{|G_f|}\le (DN)^{\frac{|G_f|}{t}}$$ if we take $D\ge f_d^t(2C+1)^tt$.
 There are two cases to consider.

\case{$\gamma\neq 0$} In this case we have 
$$0<|N_{K/\Q}(\gamma)|<(DN)^{\frac{|G_f|}t}.$$ 
Now since $p^k|N_{K/\Q}(\gamma)$
and $p>DN$, we must have $t<\frac{|G_f|}k$, proving the assertion of the lemma.

\case{$\gamma=0$} Here we need to use the assumptions (from the statement of Theorem \ref{thm: main2}) that each $S_i$ is potent and 
$\dim\langle S\cup\{1\}\rangle=d+1-r$. Throughout the rest of the proof all spans (denoted by $\langle\cdot\rangle$) and linear dependencies are over $\Q$. 

First we claim that we must have $S_i\subset\{\alpha_1,\ldots,\alpha_t\}$ for some $i$. Otherwise there exists $\beta_i\in S_i\setminus\{\alpha_1,\ldots,\alpha_t\}$ for each $1\le i\le r$ and since $S_i$ is potent $\beta_i\in\langle S_i\setminus\{\beta_i\}\rangle$, so we have
$\langle S\cup\{1\}\rangle=\langle \Cup(S_i\setminus\{\beta_i\})\cup\{1\}\rangle$, with the set in the latter span containing $\alpha_1,\ldots,\alpha_t$. But by (\ref{eq: gamma}) and the assumption $\gamma=0$ we have $\sum_{i=1}^ta_i\alpha_i=u$ and thus $\alpha_1,\ldots,\alpha_t,1$ are linearly dependent over $\Q$, so $$\dim\langle S\cup\{1\}\rangle=\dim\Big\langle \Cup_{i=1}^r(S_i\setminus\{\beta_i\})\cup\{1\}\Big\rangle\le\Bigl|\Cup_{i=1}^r(S_i\setminus\{\beta_i\})\cup\{1\}\Bigr|-1\le d-r,$$ a contradiction to the assumption $\dim\langle S\cup\{1\}\rangle=d+1-r$. This establishes our claim that $S_i\subset\{\alpha_1,\ldots,\alpha_t\}$ for some $i$.

Now suppose (after reordering the $\alpha_j$) that $S_i=\{\alpha_{1},\ldots,\alpha_{{s}}\}$. 
Then since $S_i$ is potent we have $$\sum_{l=1}^{s}b_l\alpha_{l}=w\in\Z$$ for some $0< b_l\in\Z$, and hence $\sum_{l=1}^{s}b_l\bar\alpha_{l}=\bar w$. Since $\bar\alpha_{l}=\bar n_{l}$ we obtain $\sum_{l=1}^{s}b_{l}n_{i}-w\equiv 0\pmod {p}$, which is impossible if $D\ge\sum_{l=1}^{s}b_{l}+1$, $p>DN$ and $N$ is sufficiently large, just as in the proof of Lemma \ref{lem: main1} (here we need to use the assumption that $f$ does not have integer roots). We conclude that the assertion of the lemma holds with $D=\max\left(f_d^t(2C+1)^tt,\sum_{l=1}^sb_l+1\right)$.

\end{proof}




\begin{proof}[Proof of Theorem \ref{thm: main2}] Assume the setup of Theorem \ref{thm: main2}. By Lemma \ref{lem: pre lem main2} combined with Corollary \ref{cor: sah} we may apply Proposition \ref{prop: main2} with $\delta_k=\min(d-k,\lceil e/k\rceil-1)$ and conclude that
$$\log L_f(N)\gtrsim \frac{d-1}{U_{d,e}}N\log N,$$
where $U_{d,e}$ is given by (\ref{eq: def U_e}). \end{proof}

\bibliography{mybib}
\bibliographystyle{alpha}

\end{document}